\documentclass{elsart3-1}

\usepackage{amssymb, amsmath}

\usepackage[english,francais]{babel}

\usepackage{xcolor,cancel}

\newcommand{\Ll}{\mathcal L}
\newcommand{\R}{\mathbb R}
\newcommand{\Ds}{\mathcal D}
\newcommand{\bb}{\boldsymbol b}
\newcommand{\Z}{\boldsymbol Z}
\newcommand{\RLF}{\boldsymbol X}
\newcommand{\x}{\boldsymbol x}
\newcommand{\y}{\boldsymbol y}
\newcommand{\0}{\boldsymbol 0}

\DeclareMathOperator{\diver}{div}

\newtheorem{theorem}{Theorem}[section]
\newtheorem{lemma}[theorem]{Lemma}
\newtheorem{e-claim}[theorem]{Claim}
\newtheorem{e-proposition}[theorem]{Proposition}

\newtheorem{assumption}[theorem]{Assumption}
\newtheorem{e-definition}[theorem]{Definition\rm}
\newtheorem{remark}[theorem]{\it Remark\/}

\newenvironment{proof}{\paragraph*{\bf Proof.}}{\hfill$\square$}

\renewcommand{\theequation}{\arabic{equation}}
\def\og{\leavevmode\raise.3ex\hbox{$\scriptscriptstyle\langle\!\langle$~}}
\def\fg{\leavevmode\raise.3ex\hbox{~$\!\scriptscriptstyle\,\rangle\!\rangle$}}

\journal{the Acad\'emie des sciences}
\begin{document}
\centerline{}
\begin{frontmatter}


\selectlanguage{english}
\title{Uniqueness and Lagrangianity for solutions \\ with lack of integrability of the continuity equation}


\selectlanguage{english}
\author[LC]{Laura Caravenna},
\ead{laura.caravenna@unipd.it}
\author[GC]{Gianluca Crippa}
\ead{gianluca.crippa@unibas.ch}

\address[LC]{Dipartimento di Matematica `Tullio Levi Civita', Universit\`a di Padova, via Trieste 63, 35121 Padova, Italy}
\address[GC]{Departement Mathematik und Informatik, Universit\"at Basel, Spiegelgasse 1, 4051 Basel, Switzerland}


\medskip
\begin{center}
{\small Received *****; accepted after revision +++++\\
Presented by £££££}
\end{center}

\begin{abstract}
\selectlanguage{english}
We deal with the uniqueness of distributional solutions to the continuity equation with a Sobolev vector field and with the property of being a Lagrangian solution, i.e.~transported by a flow of the associated ordinary differential equation. We work in a framework of lack of local integrability of the solution, in which the classical DiPerna-Lions theory of uniqueness and Lagrangianity of distributional solutions does not apply due to the insufficient integrability of the commutator. We introduce a general principle to prove that a solution is Lagrangian: we rely on a disintegration along the unique flow and on a new directional Lipschitz extension lemma, used to construct a large class of test functions in the Lagrangian distributional formulation of the continuity equation. 

\vskip 0.5\baselineskip

\selectlanguage{francais}
\noindent{\bf R\'esum\'e} \vskip 0.5\baselineskip \noindent
{\bf Unicit\'e et propri\'et\'e lagrangienne des solutions manquant d'int\'egrabilit\'e de l'\'equation de continuit\'e.}
On \'etudie l'unicit\'e des solutions distributionnelles de l'\'equation de continuit\'e avec des champs de vecteurs Sobolev et la propri\'et\'e d'\^etre une solution lagrangienne, c'est-\`a-dire une solution transport\'ee par le flot de l'\'equation diff\'erentielle ordinaire associ\'ee au champ de vecteurs.
On travaille dans un cadre o\`u les solutions consid\'er\'ees manquent d'int\'egrabilit\'e locale et o\`u on ne peut pas appliquer la th\'eorie classique de DiPerna-Lions d'unicit\'e des solutions distributionnelles et de la propri\'et\'e d'\^etre lagrangienne parce que on n'a pas assez d'int\'egrabilit\'e pour le commutateur.
On introduit un principe g\'en\'eral pour d\'emontrer la propri\'et\'e d'\^etre une solution lagrangienne: notre technique se base sur une desint\'egration le long le flot unique et sur un lemme d'extension lipschitzienne directionnelle qui nous permet de construire une vaste famille de fonction test pour la formulation distributionnelle lagrangienne de l'\'equation de continuit\'e.   

\end{abstract}
\end{frontmatter}


\selectlanguage{english}
\section{Introduction and statement of the main result}
\label{}

In this note we deal with the uniqueness of distributional solutions to the continuity equation with a Sobolev vector field and with the property of being a  Lagrangian solution, i.e.~transported by a flow of the associated ordinary differential equation. 

Let us first recall the by now classical DiPerna-Lions theory~\cite{DPL}. We fix $1\leq p \leq \infty$ and $T>0$ and we consider a vector field
\begin{equation}
\label{E:assumptionB}
\begin{aligned}
& \bb\in L^{1}\left([0,T];W_{\mathrm{loc}}^{1,p}(\R^{n};\R^{n})\right) \,,
\qquad \diver  \bb \in L^{1} \left( [0,T];L^\infty(\R^n) \right) \,, \\
& \frac{|\bb(t,\x)|}{1+|\x|} \in L^{1}\left( [0,T]; L^1(\R^{n}) \right)
+ L^{1}\left( [0,T] ; L^\infty(\R^{n}) \right)\,.
\end{aligned}
\end{equation}
Given an initial datum $u_{0}$, 
we consider distributional solutions to the Cauchy problem for the continuity equation
\begin{equation}
\label{E:CPPDE}
\begin{cases}
\partial_{t} u + \diver (\bb  u)  = 0
\\
 u (t=0,\x)=u_{0}(\x)
\end{cases}
\qquad\text{in $\Ds'([0,T)\times\R^{n})$,}
\end{equation}
defined as usual by a formal ``integration by parts'' after testing the equation with Lipschitz test functions. Given a vector field $\bb$ as in~\eqref{E:assumptionB}, the DiPerna-Lions theory~\cite{DPL} guarantees uniqueness of distributional solutions
\begin{equation}
\label{e:class}
u \in L^\infty \left([0,T] ; L^q(\R^n) \right)
\end{equation}
to the problem~\eqref{E:CPPDE}, where $q$ is the conjugate exponent of $p$, that is, $1/p+1/q=1$. If $u_{0} \in L^q(\R^n)$ the existence of solutions in this class can be proved by an easy approximation procedure. Moreover, such unique solution is transported by the unique regular Lagrangian flow associated to $\bb$ (see Definition~\ref{d:rlf}). We remark that the theory of~\cite{DPL} has been extended to vector fields with bounded variation by Ambrosio~\cite{AMB}.

The need for considering solutions in the class~\eqref{e:class} follows from the strategy of proof in~\cite{DPL}, which consists in showing the renormalization property for distributional solutions. To this aim, the authors prove the convergence to zero of a suitable commutator, that can be rewritten as an integral expression involving essentially the product of~$D\bb$ and~$u$. However, distributional solutions to the Cauchy problem~\eqref{E:CPPDE} can be defined as long as the product~$\bb u \in L^1_{\mathrm{loc}}([0,T] \times \R^n)$. Therefore, the theory in~\cite{DPL} leaves open the question whether uniqueness holds for solutions with less integrability than~\eqref{e:class}. Ideally, the ``extreme'' case would be that of $\bb \in L^\infty \cap W^{1,1}$ and $u\in L^1$, both locally in space.

Our main result in this direction is the following:
\begin{theorem}
\label{T:main}
Let $\bb$ be a vector field as in~\eqref{E:assumptionB}, with $1<p\leq\infty$.
Assume in addition that $\bb(t,\cdot)$ is continuous for $\Ll^{1}$-a.e.~$t \in [0,T]$, with modulus of continuity on compact sets which is uniform in time. Then, given an initial datum $u_{0} \in L^{1}_{\mathrm{loc}}(\R^{n})$, the Cauchy problem for~\eqref{E:CPPDE} has a unique solution $$
u \in L^{1}_{\mathrm{loc}} \left([0,T]\times\R^{n}\right) \,.
$$
Such unique solution is Lagrangian and renormalized. 
\end{theorem}
\begin{remark}
The continuity assumption on the vector field in Theorem~\ref{T:main} is satisfied for example when $\bb\in L^{\infty} \left([0,T];W_{\mathrm{loc}}^{1,p}(\R^{n})\right)$ with $p>n$.
\end{remark}
\begin{remark}
Theorem~\ref{T:main} can be easily extended to the case where a source term or a linear term of zero order are present in the continuity equation, under suitable integrability conditions on the coefficients. In particular, we can also deal with the transport equation \[\partial_t u + \bb \cdot \nabla u = 0\] instead of the continuity equation~\eqref{E:CPPDE}.
\end{remark}

\medskip

Let us describe in few words the strategy of the proof of Theorem~\ref{T:main}. Given a distributional solution $u \in L^{1}_{\mathrm{loc}} \left([0,T]\times\R^{n}\right)$ of the Cauchy problem~\eqref{E:CPPDE} we aim at proving that it is transported by the regular Lagrangian flow $\RLF$ associated to $\bb$. To this aim, we change variable using the flow in the distributional formulation of~\eqref{E:CPPDE}. However, due to the lack of Lipschitz regularity of the flow with respect to the space variable, we do not obtain yet the Lagrangian formulation in distributional sense: after the change of variable we do not obtain the full class of test functions.

Nevertheless some regularity of the flow ``on large sets'' is in fact available (see Theorem~\ref{T:regularity}). This guarantees that the test function we obtain is Lipschitz on a ``large flow tube'', although with a possibly large Lipschitz constant. We need to extend this function to a globally Lipschitz test function. The key remark is that, in order to estimate the error resulting from this extension, only the Lipschitz constant along the characteristics is relevant, not the global Lipschitz constant. We then implement a ``directional extension lemma'' (Lemma~\ref{l:extension}), stating that we can construct an extension which is both globally Lipschitz and directionally Lipschitz along the flow, and the directional Lipschitz constant can be estimated quantitatively. This allows to conclude the proof. 

After presenting in \S\ref{s:prelim} some background material, in \S\ref{s:proof} we give a complete proof of Theorem~\ref{T:main}, under the additional Assumption~\ref{A:srfe} on the existence of a directional Lipschitz extension. In \S\ref{S:lemmaLip} we sketch a proof of the validity of Assumption~\ref{A:srfe} under the continuity assumptions on the vector field in Theorem~\ref{T:main}. A complete proof is deferred to the follow up paper~\cite{CC2}.

\section{Some preliminaries} 
\label{s:prelim}


In the non smooth context the suitable notion of flow of a vector field is that of regular Lagrangian flow, introduced in the following form in~\cite{AMB}:
\begin{e-definition}
\label{d:rlf} 
We say that a map $\RLF:[0,T]^{2}\times\R^{n}\to\R^{n}$ is a regular Lagrangian flow associated to the vector field $\bb$ if
\begin{enumerate}
\item \label{item:1RLF} For  $(t,s) \in [0,T]^2$ we have $C^{-1}\Ll^{n}\leq \RLF(t,s,\cdot)_{\#}\Ll^{n}\leq C\Ll^{n}$.
\item For $\Ll^{n}$-a.e.~$\x\in\R^n$ the map~$\RLF(\cdot, s,\x)$ satisfies the ordinary differential equation
\begin{equation}
\label{E:ODEsRLF}
\begin{cases}
\partial_{t}\RLF(t,s,\x)=\bb(t,\RLF(t,s,\x))\\
\RLF(s,s,\x)=\x
\end{cases}
\qquad\text{in $\Ds'([0,T))$.}
\end{equation}
\end{enumerate}
\end{e-definition}
We notice that
\[
\text{$\x\mapsto\RLF(0,t,\x)$\qquad is the inverse of\qquad$\x\mapsto\RLF(t,0,\x)$.}
\]
For later use we set
\begin{equation}
\label{e:rho}
\rho(t,\cdot) \Ll^n := \RLF(t,0,\cdot)_\# \Ll^n \,, \qquad
R(t,\y) := \rho(t,\RLF(t,0,\y))
\end{equation}
and observe that by Definition~\ref{d:rlf}(i) we have 
\begin{equation}
\label{e:boundrho} 
C^{-1} \leq \rho(t,\x) \leq C\,, \qquad
C^{-1} \leq R(t,\y) \leq C\,.
\end{equation}

The theory in~\cite{DPL,AMB} guarantees that, given a vector field $\bb$ as in~\eqref{E:assumptionB}, there exists a unique regular Lagrangian flow associated to it. Moreover, in~\cite{CDL} the following regularity of the regular Lagrangian flow has been proved:
\begin{theorem}
\label{T:regularity}
Let $\bb$ be a vector field as in~\eqref{E:assumptionB} and let $\RLF$ be the associated regular Lagrangian flow. Assume that $1<p\leq\infty$. Then, for all $R>0$ and $\varepsilon>0$ there exists a compact set $K_{\varepsilon} \subset B_{R}(\0)$ such that 
\begin{enumerate}
\item $\RLF(t,s,\cdot)$ is Lipschitz continuous on $ K_{\varepsilon}$,  uniformly w.r.t.~$t,s\in[0,T]$.
\item $\Ll^{n}(B_{R}(\0)\setminus K_{\varepsilon})\leq \varepsilon$.
\end{enumerate}
\end{theorem}
The restriction to the case $p>1$ in Theorem~\ref{T:regularity} and therefore in Theorem~\ref{T:main} is due to the use of some harmonic analysis estimates in its proof.

We finally introduce the following concept of directional Lipschitz continuity:
\begin{e-definition}
\label{d:directional}
Let $\phi$ be defined on a Borel set $B \subset [0,T]\times\R^n$ and let $\Z(t,\y): [0,T]\times A\to \R^{n}$ be a Borel map, where $A \subset \R^n$ is a Borel set.
We say that the function $\phi$ is $(L,\Z)$-directionally Lipschitz continuous if for all $t,t'\in[0,T]$ and for all $\y\in A$ such that $\Z(t,\y), \Z(t',\y)\in B$ there holds
$$
|\phi(t,\Z(t,\y))-\phi(t',\Z(t',\y))|\leq L |t-t'| \,.
$$
\end{e-definition}
We focus in this paper only on directional Lipschitz continuity in the specific case $\Z(t,\y)=\RLF(t,0,\y)$, where $\RLF$ is a regular Lagrangian flow.

%
%
%

\section{Proof of Theorem~\ref{T:main}: disintegration along the regular Lagrangian flow} 
\label{s:proof}

In this section we give a complete proof of Theorem~\ref{T:main}, under the additional Assumption~\ref{A:srfe} on the existence of a directional Lipschitz extension that we introduce in Step~2 here below. A proof of Assumption~\ref{A:srfe} is sketched in \S\ref{S:lemmaLip} below and a full proof deferred to~\cite{CC2}.

\paragraph*{\textbf{Step 0.}} 
By the linearity of the continuity equation~\eqref{E:CPPDE}, it is enough to prove that $u_{0} \equiv 0$ implies $u \equiv 0$. We do this by showing that every distributional solution $u$ of~\eqref{E:CPPDE} satisfies a Lagrangian formulation. In this context this amounts to the fact that the function
\begin{equation}
\label{E:capU}
U(t,\y):=u(t,\RLF(t,0,\y))
\end{equation}
solves in distributional sense the equation $\partial_{t}[U/R]=0$, where $R$ is defined in~\eqref{e:rho}, with initial datum $U(t=0,\y)\equiv0$, that is
\begin{equation}
\label{E:lagrangian}
\int_0^T \int_{\R^{n}} \frac{U(t,\y)}{R(t,\y)} \, \partial_{t}\Psi(t,\y)\,dtd\y=0 
\qquad \text{for all test functions $\Psi(t,\y)=\Psi_{1}(t)\Psi_{2}(\y)$}\,,
\end{equation}
where $\Psi_{1}(t)\in \mathrm{Lip}_{\textrm c}([0,T))$ and $\Psi_{2} \in L^{\infty}_{\textrm c}(\R^{n})$, the spaces of Lipschitz functions with compact support, and of essentially bounded functions with compact support, respectively. Notice that the validity of~\eqref{E:lagrangian} implies that $U\equiv0$, and thus with~\eqref{E:capU} we obtain $u \equiv 0$. Since $\mathrm{Lip}_{\textrm c}(\R^{n})$ is dense in $L^{\infty}_{\textrm c}(\R^{n})$ with respect to the weak star topology of $L^\infty(\R^n)$, we reduced the proof of Theorem~\ref{T:main} to the proof of the following claim:
\begin{e-claim}
\label{c:claim}
The Lagrangian formulation~\eqref{E:lagrangian} holds for every $\Psi \in \mathrm{Lip}_{\textrm c}([0,T) \times \R^{n})$.
\end{e-claim}
We fix  
\begin{equation}
\label{e:L}
\text{a function\quad $\Psi$\quad as in Claim~\ref{c:claim} and we set \quad$L = \mathrm{Lip}(\Psi)$.}
\end{equation}
We prove in the next steps that Claim~\ref{c:claim} holds.

\paragraph*{\textbf{Step 1.}} Fix $\varepsilon>0$ and consider a compact set of the form $[0,T]\times B_{R}(\0)$ which contains the support of the function $\Psi$ fixed in~\eqref{e:L}. We use Theorem~\ref{T:regularity} to find a compact subset $K_{\varepsilon} \subset B_{R}(\0)$ on which the regular Lagrangian flow $\RLF(t,s,\cdot)$ is uniformly Lipschitz continuous. 

\begin{lemma}
\label{L:sfsdva}
On the compact flow tube $\{\RLF(t,0,K_{\varepsilon})\}_{t\in[0,T]}$ starting from $K_{\varepsilon}$ the function
\begin{equation}
\label{E:testtx}
\psi(t,\x):=\Psi(t,\RLF( 0, t,\x))
\qquad
\text{for $ \x \in \RLF(t,0,K_{\varepsilon})$}
\end{equation}
is Lipschitz continuous and $(L,\RLF)$-directionally Lipschitz continuous, with $L$ as in~\eqref{e:L}.
\end{lemma}

\begin{proof} We start by proving the $(L,\RLF)$-directional Lipschitz continuity. Let
$$
\x=\RLF(t,0,\y) \; \text{ and } \; \x'=\RLF(t',0,\y)\,,
\quad
\text{ so that by~\eqref{E:testtx}} 
\quad
\psi(t,\x)=\Psi(t,\y)
\; \text{ and } \;
\psi(t',\x')=\Psi(t',\y) \,, 
$$
and thus by~\eqref{e:L} we get
\[
|\psi(t,\x)-\psi(t',\x')|=|\Psi(t,\y)-\Psi(t',\y)|\leq \mathrm{Lip}(\Psi) |t-t'|=L|t-t'|\,.
\]
We now prove the Lipschitz continuity of $\psi$ on $\{\RLF(t,0,K_{\varepsilon})\}_{t\in[0,T]}$. Given $\x,\x^{*}\in\RLF(t,0,K_{\varepsilon})$ one has
\[
|\psi(t,\x)-\psi(t,\x^{*})|
= |\Psi(t,\RLF( 0, t,\x))-\Psi(t,\RLF( 0, t,\x^{*}))|
\leq \mathrm{Lip}(\Psi) \cdot \mathrm{Lip}(\RLF(0,t,\cdot)|_{K_{\varepsilon}})\, |\x-\x^{*}|\,.
\]
When comparing two points $\x\in\RLF(t,0,\y)$ and $\x'\in\RLF(t',0,\y')$, for some $\y,\y'\in K_{\varepsilon}$, we simply define $\x^{*}=\RLF(t,0,\y')$ and we estimate
\begin{align*}
|\psi(t,\x)-\psi(t',\x')|&\leq|\psi(t,\x)-\psi(t,\x^{*})|+|\psi(t,\x^{*})-\psi(t',\x')|
\\
&\leq  C(\mathrm{Lip}(\Psi),\mathrm{Lip}(\RLF(0,t,\cdot)|_{K_{\varepsilon}}),\lVert\bb\rVert_{\infty})
		\cdot(|\x-\x^{*}| + |t-t'| )
\\
&\leq  C(\mathrm{Lip}(\Psi),\mathrm{Lip}(\RLF(0,t,\cdot)|_{K_{\varepsilon}}),\lVert\bb\rVert_{\infty})
		\cdot(|t-t'|+|\x-\x'|)\,,
\end{align*}
where in the last inequality we applied 
\[
|\x-\x^{*}|\leq |\x-\x'|+|\x'-\x^{*}| 
\quad \text{ and } \quad
|\x^{*}-\x'|\leq\lVert\bb\rVert_{\infty}|t'-t|\,.
\]
This concludes the proof of the lemma. 
\end{proof}

\paragraph*{\textbf{Step 2.}}
We can proceed with the proof under the following assumption. 
\begin{assumption}
\label{A:srfe}
Given $\varepsilon>0$ let $\psi$ be as in~\eqref{E:testtx}. We assume that there exists $\psi_{\varepsilon} : [0,T] \times \R^{n}\to\R$ which is an extension of $\psi$ and in addition is
\begin{enumerate}
\item Lipschitz continuous, and 
\item $(L',\RLF)$-directionally Lipschitz continuous, where $L'>0$ does not depend on $\varepsilon$.
\end{enumerate}
\end{assumption}

In fact, we are able to prove that Assumption~\ref{A:srfe} holds when the vector field $\bb$ satisfies the continuity condition assumed in Theorem~\ref{T:main}. In Section~\ref{S:lemmaLip} we give a sketch of the proof of this fact, and we defer a complete proof to a next paper.

\paragraph*{\textbf{Step 3.}}
We now derive some consequences of Assumption~\ref{A:srfe} in the $(t,\y)$-variables. We define
\begin{equation}
\label{e:capPsi}
\Psi_{\varepsilon}(t,\y) := \psi_{\varepsilon}(t,\RLF(t,0,\y))
\qquad
\text{for $t \in [0,T]$ and $\x \in \R^n$}
\end{equation}
and we observe that 
\begin{enumerate}
\item \label{i:szvds} $\Psi_{\varepsilon}(\cdot,\y)$ is $L'$-Lipschitz continuous for all $\y$.
This follows from Assumption~\ref{A:srfe}(ii) and from the definition of directional Lipschitz continuity (Definition~\ref{d:directional}).
%
%
\item $\Psi_{\varepsilon}(t,\y)\equiv \Psi_{}(t,\y)$ for every $\y\in K_{\varepsilon}$ and every $t \in [0,T]$.
\end{enumerate}
In particular, we can test $\partial_{t}[U/R](t,\y)$ agains $\Psi_{\varepsilon}(t,\y)$: by the definitions in~\eqref{E:capU} and~\eqref{e:capPsi} we obtain
\begin{align*}
\int_0^T \int_{\R^n} &\frac{U(t,\y)}{R(t,\y)} \, \partial_{t}\Psi_{\varepsilon}(t,\y)\,dtd\y
=
\int_0^T \int_{\R^n} \frac{u(t,\RLF(t,0,\y))}{\rho(t,\RLF(t,0,\y))} \, \frac{d}{dt}\psi_{\varepsilon}(t,\RLF(t,0,\y)) \,dtd\y
\\
&=
\int_0^T \int_{\R^n} \frac{u(t,\RLF(t,0,\y))}{\rho(t,\RLF(t,0,\y))} \,
\left[\big(\partial_{t}\psi_{\varepsilon}\big)(t,\RLF(t,0,\y))+
\bb(t,\RLF(t,0,\y))\cdot\big(\nabla \psi_{\varepsilon}\big)(t,\RLF(t,0,\y))\right]\,dtd\y\,.
\end{align*}
We now apply the change of variable $ \x=\RLF(t,0,\y)$, obtaining 
\begin{equation}
\label{E:PDEe}
\int_0^T \int_{\R^n} \frac{U(t,\y)}{R(t,\y)} \, \partial_{t}\Psi_{\varepsilon}(t,\y)\,dtd\y
=
\int_0^T \int_{\R^n} u(t,\x)\left[\partial_{t}\psi_{\varepsilon}(t,\x)+\bb(t,\x)\cdot\nabla\psi_{\varepsilon}(t,\x)\right] \,dtd\x=0\,,
\end{equation}
because $u$ is a distributional solution of~\eqref{E:CPPDE} with zero initial datum.
We stress that the first equality in~\eqref{E:PDEe} follows by the definition of push-forward measure because the results in~\cite{DPL} establish that the {regular} Lagrangian flow $\RLF$ satisfies the absolute continuity estimate in Definition~\ref{d:rlf}\eqref{item:1RLF}.
This is a very important brick in this disintegration strategy, and in other settings it requires to be proved ad hoc, see for instance~\cite{CD}.

\paragraph*{\textbf{Step 4.}}
We conclude the proof of Claim~\ref{c:claim}, thus establishing Theorem~\ref{T:main} under Assumption~\ref{A:srfe}. The main observation is that equation~\eqref{E:PDEe} gives the validity of Claim~\ref{c:claim} with the test function $\Psi$ replaced by the approximation $\Psi_\varepsilon$ defined in~\eqref{e:capPsi}. Therefore, we simply estimate the integral containing $\Psi$ with the integral containing $\Psi_{\varepsilon}$ plus an error, and we only need to show that the error converges to zero as $\varepsilon\downarrow0$. Indeed, we compute as follows:
%
\begin{align*}
\int_0^T \int_{\R^n} \frac{U}{R}\,\partial_{t}\Psi\,dtd\y &= 
\cancel{\int_0^T \int_{\R^n} \frac{U}{R}\,\partial_{t}\Psi_{\varepsilon}\,dtd\y}
+\int_0^T \int_{\R^n} \frac{U}{R}\,\partial_{t}\left[\Psi-\Psi_{\varepsilon}\right]\,dtd\y 
\\
&=\cancel{\int_{0}^{T}\int_{K_{\varepsilon}} \frac{U}{R}\,\partial_{t}\left[\Psi-\Psi_{\varepsilon}\right]\,dtd\y} 
	+\int_{0}^{T}\int_{ B_{R}(\0)\setminus K_{\varepsilon} }  \! \frac{U}{R}\,\partial_{t}\left[\Psi-\Psi_{\varepsilon}\right]\,dtd\y \,,
\end{align*}
where $K_\varepsilon$ is as in Step~1, and by construction $\Psi_\varepsilon \equiv \Psi$ on $[0,T] \times K_\varepsilon$. 
Since $\Psi_{}(\cdot,\y)$ is $L$-Lipschitz continuous by definition~\eqref{e:L} and each $\Psi_{\varepsilon}(\cdot,\y)$ is $L'$-Lipschitz continuous by Step~3(\ref{i:szvds}), we finally get 
\begin{align*}
\left| \int_0^T \int_{\R^n} \frac{U}{R}\,\partial_{t}\Psi\,dtd\y \right|
\leq 
C(L+L')\int_{0}^{T}\int_{ B_{R}(\0)\setminus K_{\varepsilon} }  \!\!\!\! \!\!\!\! \left| U \right|\,dtd\y 
\quad\xrightarrow{\varepsilon\downarrow0}\quad 0\,, 
\end{align*}
using~\eqref{e:boundrho} and the fact that the function $U$ in~\eqref{E:capU} belongs to $L^1_{\mathrm{loc}}([0,T] \times \R^n)$. This concludes the proof of Theorem~\ref{T:main} under Assumption~\ref{A:srfe}.

\section{Idea of the proof of Assumption~\ref{A:srfe}: directional Lipschitz extension lemma}
\label{S:lemmaLip}


We finally briefly sketch the strategy of proof of the following lemma. A full proof in a more general context is deferred to~\cite{CC2}.
\begin{lemma}
\label{l:extension}
Let $\bb$ be a vector field as in~\eqref{E:assumptionB}, with $1<p\leq\infty$. 
Assume in addition that $\bb(t,\cdot)$ is continuous for $\Ll^{1}$-a.e.~$t \in [0,T]$, with modulus of continuity on compact sets which is uniform in time. Then Assumption~\ref{A:srfe} holds.
\end{lemma}
In the above lemma one can as well require that $\min \psi \leq \psi_{\varepsilon} \leq \max \psi$.

We start by noticing that a function $\psi$ is $(L,\RLF)$-directionally Lipschitz continuous according to Definition~\ref{d:directional} if and only if $\psi$ is $L$-Lipschitz continuous for the following degenerate distance $d_{0}$: 
$$
d_{0}\left( (t,\x), (t',\x') \right) :=\begin{cases}|t-t'| & \text{if there exists $\y$ with $\x= \RLF(t,0,\y)$ and $\x'= \RLF(t',0,\y)$,} \\ +\infty &\text{otherwise.}\end{cases}
$$
Moreover, we denote by $d_{1}$ the usual Euclidean distance in $[0,T] \times \R^{n}$.

Consider a Lipschitz continuous function $\psi$ defined on the compact flow tube $\{\RLF(t,0,K_\varepsilon)\}_{t\in[0,T]}$ of Assumption~\ref{A:srfe}. We remind that we assume that $\psi$ is $(L,\RLF)$-directionally Lipschitz continuous, and that we need to extend $\psi$ to $[0,T] \times \R^n$ in such a way that the extension is
\begin{enumerate}
\item $(L,\RLF)$-directionally Lipschitz continuous, i.e.~$L'$-Lipschitz continuous for $d_{0}$, with $L'$ depending on~$L$, and
\item Lipschitz continuous for the Euclidean distance, i.e.~Lipschitz continuous for $d_{1}$.
\end{enumerate}
In other words, we need to prove a Lipschitz extension theorem with respect to two non equivalent distances at the same time: to the best of our knowledge, this is a new and non trivial task. Notice that for our purposes we need that the Lipschitz constant for $d_0$ only depends on $L$, while we do not need a quantitative control on the Lipschitz constant for $d_1$.

We now give a rough idea of the proof of Lemma~\ref{l:extension}. For $0<\lambda<1$ we introduce a family of distances~$d_{\lambda}$, each of them equivalent to the Euclidean distance $d_1$.  The distance $d_{\lambda}$ penalizes with a factor $\lambda^{-1}$ displacements which are not along the flow.
Moreover, the distances $d_\lambda$ converge to the degenerate distance $d_{0}$, i.e.~$d_{\lambda} \uparrow d_{0}$ as $\lambda\downarrow0$.
In particular, a function which is $L'$-Lipschitz continuous for $d_{\lambda}$ is also $L'$-Lipschitz continuous for~$d_{0}$.

The key point in the proof of Lemma~\ref{l:extension} is the fact that, when $\lambda\downarrow0$, the Lipschitz constant of~$\psi$ for~$d_{\lambda}$ converges to the Lipschitz constant $L$ of $\psi$ for $d_{0}$:
\begin{equation}
\label{E:conv}
L_{\lambda}:=\mathrm{Lip}_{}(\psi;d_{\lambda}) \xrightarrow{\lambda\downarrow0} \mathrm{Lip}_{}(\psi;d_{0}) =  L \,.
\end{equation}
Using this property, we choose $\bar\lambda$ small enough so that $L_{\bar\lambda}$ is  close to $L$. We extend $\psi$ by using McShane extension theorem for the distance $d_{\bar \lambda}$. In this way, we get an extension  which is
\begin{enumerate}
\item $(L_{\bar\lambda}, \RLF)$-directionally Lipschitz continuous, and $L_{\bar \lambda}$ is close to $L$, and
\item Lipschitz continuous for the Euclidean distance $d_{1}$, since $d_{\bar \lambda}$ is equivalent to $d_1$.
\end{enumerate}

In the above procedure, we are currently able to prove~\eqref{E:conv} only assuming that the vector field $\bb$ is continuous for $\Ll^{1}$-a.e.~$t \in [0,T]$, with modulus of continuity on compact sets which is uniform in time.

\section*{Acknowledgements}
This work was started during a visit of LC at the University of Basel and carried on during a visit of GC at the University of Padova as a Visiting Scientist. The authors gratefully acknowledge the support and the hospitality of both institutions. LC is a member of the Gruppo Nazionale per l'Analisi Matematica, la Probabilit\`a e le loro Applicazioni (GNAMPA) of the Istituto Nazionale di Alta Matematica (INdAM). GC is partially supported by the ERC Starting Grant 676675 FLIRT.




\end{document}